\documentclass{amsart}

\usepackage{amsfonts}
\usepackage{amssymb}
\usepackage{amsthm}
\usepackage{amsmath}

\usepackage[english]{babel}

\usepackage{xcolor}

\usepackage[colorlinks=true, allcolors=blue]{hyperref}

\usepackage{tikz}
\usepackage{tikz-qtree}

\newtheorem{theorem}{Theorem}[section]

\newtheorem{lemma}[theorem]{Lemma}
\newtheorem{proposition}[theorem]{Proposition}
\newtheorem{defn}[theorem]{Definition}

\newtheorem{rmk}[theorem]{Remark}

\newcommand{\R}{\mathbb R}
\newcommand{\C}{\mathbb C}

\newcommand{\Z}{\mathbb Z}
\newcommand{\cl}{\mathcal {L}}
\begin{document}

\title{On the spectra of Cantor measures}
\author{Leandro Zuberman
}
\date{}
\address{Departamento de Matem\'atica. Universidad Nacional de Mar del Plata.\\ Mar del Plata. PBA 7600 ARGENTINA}
\email{leandro.zuberman@gmail.com}



\begin{abstract}
We consider Cantor measures on the line, with contraction factor $N^{-1}=p^{-\alpha}$ (where $p$ a positive prime, $\alpha$ a positive integer) and $m$ positive integer digits lying in distinct residue classes modulo $N$. 
We obtain a complete characterization of maximal orthogonal sets of exponentials in $L^2(\mu)$, for a class of such measures $\mu$. It is proved that the
$n+1$-th digit in the base-$N$ expansion of frequencies in a maximal orthogonal set, with the first $n$ digits  prescribed, has $m$ possible values. In consequence, there are a correspondence between labelings of the $m$-homogeneous rooted tree and maximal orthogonal sets of frequencies. 
\end{abstract}


\maketitle

\section{Introduction}

For certain measures $\mu$, there exists a discrete set $\Lambda$ satisfying that the set $E(\Lambda):=\{e^{2\pi i\lambda x}:\lambda \in \Lambda\}$ is an orthonormal basis for the Hilbert space $L^2(\mu)$. In this case, the measure is said to be spectral and the set $\Lambda$ is called a spectrum.

Spectral measures generalize spectral sets. A set $\Omega\subset \R^d$ is said to be spectral if there is an orthonormal basis of the type $E(\Lambda)$ for $L^2(\Omega)$. Fuglede conjectured that spectral sets are those sets tilling the space $\R^d$ by translations \cite{Fug74}. Although the conjecture has been proved to be false in full generality \cite{Tao04}, it has caught the interest of the community and has been shown to be valid in many particular cases. For instance, recently it has been  proved for convex sets \cite{LM19}.

Self similar Tilings have been thoroughly studied, and deep results have been obtained in this direction \cite{GH94,LW96}. In an interplay between self-similar tilings and Fuglede's conjecture, Jorgensen and Pedersen gave an example of a non-atomic singular measure which is spectral \cite{JP98}. In fact, they constructed a spectrum for a Cantor measure on the line with contraction ratio 1/4. This result kept the atention of the fractal community and was generalized in many ways. For instance, Strichartz proved that some measures with a product structure are spectral and exhibit a spectrum based on this product structure \cite{Str00}. 
Laba and Wang proved that Cantor self similar measures on the line (see definition below \eqref{autosim}) with digits admitting a Hadamard pair (see Definition \ref{Had_pair} below) are spectral \cite{LW02}. Dutkay, Haussermann and Lai extended this result to Euclidean spaces \cite{DHL19}.
These results reveal a vast family of Cantor measures that are spectral.

Spectral sets or measures can, of course, have many spectra. A natural question is how the frequencies of the spectra are distributed in the line or in the space. A classical result of Landau \cite{Lan67} gives necessary conditions for a discrete set to be a spectrum for a set. 
Dutkay et al \cite{DHSW11} extended Landau's result to spectral singular measures. They 
proved that certain spectra of Cantor measures necessarily have Beurling dimension 
equal to the dimension of the measure itself. They conjectured that this could be extended to any spectrum. However, An and Lai \cite{AL23} proved that any self similar measure whose digits have a Hadamard pair has a spectrum with Beurling dimension zero. Ultimately, little is known about the structure or distribution of spectra of singular measures.

Dutkay, Han and Siu \cite{DHS09} revisited the seminal example by Jorgensen and Pedersen (the Cantor measure with contraction factor 1/4 and two digits, giving dimension $1/2=\log 2/\log 4$) and studied all the possible spectra in this case. 
They obtained a complete characterization of all maximal orthogonal sets of type $E(\Lambda)$.
This was achieved by labeling the 2-homogeneous rooted tree using the base-4 expansion of the elements of $\Lambda$. In fact, they proved that at every position, the digits in the expansion have only two allowed values. Thus, the first $n$ digits have $2^n$ possibilities out of the $4^n$ possible. This can be thought of as a dimensional necessary condition for the spectra, coinciding with the dimension of the measure.  

In this article, we extend this result from a single measure to a family of Cantor measures. 
Specifically, we consider Cantor measures $\mu$ on the line with a contraction factor equal to the inverse of an integer that is a prime power, say $N=p^\alpha$, and digits in a finite set $D$ of positive integers in different residue classes modulo $N$. 
We prove that if $E(\Lambda)$ is orthogonal, then for each element of $\Lambda$ whose first $n$ digits in the base-$N$ expansion 
are prescribed, the number of possible digits at position $n+1$ is bounded by $|D|$ (Theorem \ref{orto}). 
 As a consequence, each such set $\Lambda$ can be arranged in a homogeneous rooted tree. Furthermore, we analyze maximal orthogonal sets of type $E(\Lambda)$. Under additional hypotheses on the aforementioned Cantor measures, we prove that the bound on the number of digits is always attained. That is, if we look at digit $n+1$ in the base-$N$ expansion of elements of $\Lambda$ with the first $n$ digits  prescribed, we will find $|D|$ possible values.  
  Moreover, if $L$ is the set of such digits at position $n+1$, then $(N,D,L)$ is a Hadamard triple (see \ref{Had_pair} for the definition of a Hadamard triple and Theorem \ref{Hadamard}).  
  Finally, we  establish a correspondence between sets $\Lambda$ for which $E(\Lambda)$ is a maximal orthogonal set and homogeneous rooted trees (with homogeneity equal to $|D|$). In fact, we prove that for each such $\Lambda$ there is a labeling of that tree and conversely (Theorem \ref{theotree}). As in the original result, this can be thought of as a dimensional necessary condition for $\Lambda$ to be spectral. 

Let us briefly comment on the hypothesis assumed and proofs. Orthogonality in $L^2(\mu)$ is naturally connected to the zeros of the Fourier transform of the measure, $\hat\mu$. As a consequence of the self-similarity, the Fourier transform of Cantor measures can be expanded as a product (see \eqref{producto}), which helps with finding its zeros. This product expansion, has factors related to the polynomial $P_D(x)=\sum_{d\in D}x^d$. In fact, a hypothesis assumed on the measure (actually on its digits) is that $P_D$ must be a product of cyclotomic polynomials (Theorems \ref{Hadamard} and \ref{theotree}) . The scale four measure considered in \cite{JP98, DHS09} has digits satisfying this condition. Although the results presented in this article extend those obtained in \cite{DHS09}, the proofs are quite different. Many of them are based on ideas appearing in \cite{Lab02}.

This article is structured as follows. In Section \ref{setting} we give most of the required definitions and give context. In Section \ref{main} we present the statements and proofs of the main results bounding and estimating the number of digits in $N$-expansions. In Section \ref{tree} we construct labelings for the rooted $m$ homogeneous trees and prove the correspondence with maximal orthogonal sets.

\section{Setting}\label{setting}
\subsection{Cantor measures}
In this paper we study Cantor-type measures on the line which are attractors of an IFS given by contraction of the form $\phi_d(x):=N^{-1}(x+d)$. More specifically, for a given integer $N>1$ and a finite set $D\subset\Z$, with elements in different residue classes modulo $N$, we consider the unique probability measure satisfying:
\begin{equation}\label{autosim}
\mu
=\frac{1}{N}\sum_{d\in D}\mu \circ \phi_d^{-1}.
\end{equation}
We call digits to the elements of $D$ and contractive factor to $N^{-1}$ or just $N$.

We adopt the notation $e(x)=e^{2\pi i x}$ and $e_\xi(x)=e(x\cdot\xi)=e^{2\pi i \xi\cdot x}$.

The Fourier transform of a finite measure $\mu$ is defined by $\hat\mu(\xi)=\int e_\xi(x)d\mu(x)$. For measures 
satisfying \eqref{autosim} we have:
\[\hat\mu(\xi)= m_D \left(\frac{\xi}{N} \right)\hat\mu \left(\frac{\xi}{N}\right),\]
where $m_D(t)=N^{-1}\sum_{d\in D}e_d(t)$. The function $m_D$ has been used in many different context such as wavelets or self-similar tilings and is referred to as filter, mask or symbol. Inductively condition \eqref{autosim} gives: 
\begin{equation}\label{producto}
\hat\mu(\xi)=\prod_{j=1}^\infty m_D(N^{-j}\xi).
\end{equation}

Observe that $e_\lambda$ and $e_{\lambda'}$ are orthogonal in $L^2(\mu)$ if and only if $\hat \mu(\lambda-\lambda')=0$. Other important observation (derived from \eqref{producto}) is that $\lambda$ is a zero of $\hat\mu$ when there is a $j\geq 1$ such that $N^{-j} \lambda$ is a zero of $m_D$.

\subsection{Cyclotomic polynomials}\label{PD}

Associated to $D$ (when $D\subset \Z_{\geq 0}$), we will consider the polynomial $P_D(x)=\sum_{d\in D}x^d$. Note that, $P_D(e^{2\pi i \xi})=N \cdot m_D(\xi)$. Thus, (real) zeros of $m_D$ are related to (complex) zeros of $P_D$ with modulus one. Those zeros are very connected to cyclotomic polynomials. 

We denote by $\Phi_n$ the cyclotomic polynomial of order $n$, which is the minimal polynomial of a primitive $n$th root of the unit $e(k/n)$ with $k$ and $n$ coprime. In other words, we have:
\begin{equation*}
    \Phi_n(x) = \prod_{\substack{{gcd(k:n)=1,}\\\\{ 1\leq k\leq n} }}(1-e(k/n)).
\end{equation*}
It will be useful to recall that if $p$ is prime $\Phi_p(x)=1+x+\cdots+x^{p-1}$ and $\Phi_{p^\alpha}(x)=\Phi_p(x^\alpha)$. In particular, $\Phi_{p^\alpha}(1)=p$.

From now on we fix $N=p^\alpha$ (where $p$ is a positive prime number, and $\alpha$ a positive integer). Also fix $D\subset \Z_{\geq 0}$ a finite set of positive
integers in different residue classes modulo $N$. Observe that this forces the IFS $\{ \phi_d \}_{d\in D}$ to have no overlaps. In addition, $\mu$ will from now on denote the probability measure defined by \eqref{autosim}. 

We will often consider the set $\{t>0:\Phi_{p^t}\mid P_D\}$. We will use from now on the letter $T$ to designate this set. Observe that we necessarily have $\prod_{t\in T}\Phi_{p^t}\mid P_D$. In consequence, $p^{|T|}$ is a factor of $|D|$ (evaluating at 1). In particular, $p^{|T|}\leq |D|$. Another important observation is that, if an element $t\in T$ were strictly bigger than $\alpha$ then the polynomial $P_D$ would have different terms with powers in the same residue class modulus $N$. Since the powers in $P_D$ are the elements of $D$ and we are assuming $D$ has all its element in different residue classes modulo $N$, this can not be the case. Therefore, $T$ will be bounded by $\alpha$.

\subsection{Expansions on base \texorpdfstring{$N$}.}
To understand the integer spectra of $\mu$, it will be usefull to expand the integers in base $N$. So, for the fixed $N>1$ and a given $k\in \Z$ we will define two sequences of integers, $k_n$ and $d_n$, with $0\leq d_n<N$ and $n\geq 0$. We start by $k_0:=k$ and the relation $k_0=d_0+Nk_1$ defines univocally $d_0$ and $k_1$. Now, inductively, if $k_n$ is given, the condition $k_{n}=d_n+Nk_{n+1}$ defines $d_n$ and $k_{n+1}$. Observe that if $k\geq 0$ then $0\leq k_{n+1}\leq \frac{k_n}{N}$ for any $n\geq 0$. So, $k_n$ is a decreasing sequence of nonnegative integers and, in consequence, there is an $n_0$ such that $k_n=0$ for $n\geq n_0-1$ and $d_n=0$ for $n\geq n_0$. Similarly, if $k<0$, there is $n_0$ such that $k\geq -N^{n_0}$. So that, $-N^{n_0-n}\leq k_n< 0$ for $0\leq n\leq n_0$. In particular, $k_{n_0}=-1$ and $d_n=N-1$ for any $n\geq n_0$. 

In this way for any integer $k$ we generate a (unique) sequence $(d_n(k))_{n\geq 0}$ with $0\leq d_n(k)<N$ ended either by only zeros or only by $(N-1)$'s. More over, there is $n_0$ such that $k=\sum_{n=0}^{n_0} d_n(k)N^k$ if $k\geq 0$ or $k=\sum_{n=0}^{n_0} d_n(k)N^k-N^{n_0}$ if $k<0$. The $n_0$ is not unique. Through the paper, we will refer to this sequence when we write $(d_n(k))_{n\geq 0}$.

It will be useful to define the following subsets of a given set $\Lambda\subset \Z$. Given $n\geq 0$ and a finite sequence of integers $(d_j)_{j=0}^{n-1}$ with $0\leq d_j<N$ 
define:
\[
\Lambda_n(d_0,\cdots,d_{n-1})=\{\lambda \in \Lambda: d_j(\lambda)=d_j \quad\forall 0\leq j<n\}.
\] If $n=0$, we interpret $\Lambda_0=\Lambda_0(\emptyset)=\Lambda$.

For a finite set $A$, we denote by $|A|$  its number of elements.

\section{Main Results}\label{main}
Recall that we have fixed $N=p^\alpha$ ($p$ positive prime, $\alpha>0$), $D\subset\Z_{\geq 0}$ finite subset with elements in different residue classes modulo $N$, 
$P_D(x)=\sum_{d\in D}x^d$, $\mu$ the probability measure satisfying \eqref{autosim},
$m_D(\xi)=N^{-1}P_D(e^{2\pi i\xi})$ and  $T=\{t>0:\Phi_{p^t}\mid P_D\}$. 
\subsection{Bounding branches}
In the next proposition we prove that in an orthogonal set of frequencies, if the first $n$ digits in the base-$N$ expansion are prescribed, then the possibilities for the $(n+1)$-th digit are bounded. 
The proof is closely related to arguments from \cite{Lab02}.

\begin{theorem}\label{orto}
	Let $\Lambda$ be a set of integers. If $E(\Lambda)$ is an orthogonal set in $L^2(\mu)$ then for any $n\geq 0$ and any finite sequence of integers $(d_j)_{j=0}^{n-1}$ with $0\leq d_j<N$, we have:
\[
	|\{d_n(\lambda):\lambda\in\Lambda_n(d_0,\cdots,d_{n-1})\}|\leq p^{|T|}. 
\]

\end{theorem}

\begin{proof}
	Fix $n\geq 0$ and $(d_j)_{j=0}^{n-1}$ as in the hypotheses. Consider $\{\lambda_0, \cdots, \lambda_{m-1} \}\subset \Lambda_n(d_0,\cdots,d_{n-1})$ with
	$d_n(\lambda_j) \neq d_n(\lambda_l)$ if $j\neq l$. We want to see that $m\leq p^{|T|}$. 
    
	Since $E(\Lambda)$ is an orthogonal set, if $j\neq l$, we have 
	$$\langle e_{\lambda_j},e_{\lambda_l}\rangle _{L^2(\mu)} = \hat\mu(\lambda_j-\lambda_l) = 0.$$ 
	In virtue of \eqref{producto}, there is a $k_{jl}$ such that $m_D((\lambda_j-\lambda_l)N^{-k_{jl}})=0$.
	Define the following numbers in the interval (0,1): 
	\begin{equation}\label{tita}
	\theta_{jl} = \frac{\lambda_j-\lambda_l}{ N^{n_{jl}}} - \lfloor \frac{\lambda_j-\lambda_l}{ N^{n_{jl}}}\rfloor,
	\end{equation}
	 where $\lfloor t \rfloor$ denotes the greatest integer less or equal than $t$. We can simplify $\theta_{jl} = \frac{s_{jl}}{p^{t_{jl}}}$ with 
	 $p\nmid s_{jl}$ and $t_{jl}\leq \alpha n_{jl}$. 
	 Since $m_D$ is $\Z$-periodic, we have:
	 \[
	 m_d(\frac{s_{jl}}{p^{t_{jl}}})=m_D(\theta_{jl}) = m_D ((\lambda_j-\lambda_l)N^{-n}) =0.
	 \]
	 So, $P_d(e^{ 2\pi i \theta_{jl}} )=0$. 
	 
	 Since $s_{jl}$ is not divisible by $p$, $e^{2\pi i \theta_{jl}}$ is also a root of $\Phi_{p^{t_{jl}}}$. Moreover, is its minimal polynomial. Then we can conclude that $\Phi_{t_{jl}} \mid P_D$. If we consider the set of all possible exponent $t_{jl}$, say, $\tilde T:=\{t_{jl}:j\neq l\}$, then we have $\tilde T\subset T$ and $\prod_{t\in \tilde T}\Phi_t \lvert P_D$. 
	 
	 Observe that 
	 $
	 \lambda_j-\lambda_l =N^{k_{jl}-1}p^{\alpha-t_{jl}}s_{jl}.
	 $
	As we pointed out in the previous section, $T$ is bounded by $\alpha$ since the elements of $D$ belong to different residue classes modulo $N$. Thus, the factor $p^{\alpha-t_{jl}}s_{jl}$ is an integer not divisible by $N$. On the other hand,
	we have that $N^n \mid (\lambda_j - \lambda_l)$ since they are in $\Lambda_n (d_0,\cdots,d_{n-1})$ but $N^{n+1} \nmid (\lambda_j - \lambda_l)$ (if $j\neq l$) since $d_n(\lambda_j)\neq d_n(\lambda_l)$. So, we can conclude that all $k_{jl}=n+1$ and $\theta_{jl} = \frac{d_n(\lambda_j)-d_n(\lambda_l)}{N}$. 

	Now, define $\theta_{jj}=0$ and $\Theta=\{e(\theta_{j0}):0\leq j <m\}$. Consider the polynomial $F(x)=(x-1)\displaystyle\prod_{t\in T}\Phi_t(x)$. Since $\theta_{j0} - \theta_{l0} = \theta_{jl}$, we have that any pair of elements $\zeta, \zeta'$ of $\Theta$ satisfy:
	\begin{equation}\label{ciclo}
	F(\zeta)=0 \qquad \text{ and }\qquad F(\zeta/\zeta')=0. 
	\end{equation} 
	We can conclude that $m\leq p^{|\tilde T|}\leq p^{|T|}$ (and so the proof) if we can prove that the number of complex roots of $F$ satisfying \eqref{ciclo} is bounded by $p^{|\tilde T|}$. We prove that polynomial property in Lemma \ref{raices} below. 
		
\end{proof}

\begin{rmk}\label{ll}
The following facts were established in the previous proof. We note it here for future reference. 
\begin{enumerate}
    \item Let $\lambda$ and $\lambda'$ be integers such that $e_\lambda$ and $e_{\lambda'}$ are orthogonal. Assume $d_k (\lambda) = d_k(\lambda') $ for $0\leq k<n$ and $d_n(\lambda)\neq d_n(\lambda')$. Then $m_D(\frac{\lambda-\lambda'}{N^n})=0$. In particular, $\frac{\lambda-\lambda'}{N^n}\notin\Z$ and $\frac{\lambda-\lambda'}{N^n}=\frac{s}{p^t}$ with $p\nmid s$ and $t\in T$ 
    \item If $l,l'\in \{d_n(\lambda):\lambda\in\Lambda_n(d_0,\cdots,d_{n-1})\}$, with $l\neq l'$ then $m_D(\frac{l-l'}{N})=0$. 
    \item $|\{d_n(\lambda):\lambda\in\Lambda_n(d_0,\cdots,d_{n-1})\}|\leq |D|.$

\end{enumerate}
\end{rmk}

The following lemma, which completes the proof of Theorem \ref{orto}, was proved in \cite{Lab02}. We repeat it here for the sake of completeness.

\begin{lemma}\label{raices}
Let $\beta_1,\cdots,\beta_m$ be diferent positive integers. Let $p$ be a positive prime number. Let $F=F_{0,\beta_1,\cdots,\beta_m}$ be the following polynomial:
\[
F(x)=(x-1)\prod_{j=1}^m \Phi_{p^{\beta_j}}(x).
\]
If $\zeta_1,\cdots,\zeta_n$ are roots of $F$ satisfying $F(\zeta_j/\zeta_l)=0$ then $n\leq p^m$. 
\end{lemma}

\begin{proof}
	The proof is by induction in $m$. If $m=1$ then $F(x)=(x-1)\Phi_{p^{\beta_1}}(x)$ and its roots are of the type $e^{2\pi i k/p^{\beta_1}} $ for some $k\in \Z$. If $p\mid k-k'$ then either $e^{2\pi i k/p^{\beta_1}} $ and $e^{2\pi i k'/p^{\beta_1}} $ were the same root or $e^{2\pi i (k-k')/p^{\beta_1}} $ is not a root of $F$. So, the number of roots satisfying the mentioned property is bounded by the number of class modulus $p$ which is $p^1$, as we wish. 
	
	Now, assume the thesis is valid for $m-1$. Again, assume $\zeta_j=e^{2\pi i k_j}$. For $0\leq l <p$ define $Z_l=\{\zeta_j:k_j\equiv l(p) \}$. Observe that using the same argument as before, if $\zeta_j,\zeta_{j'}$ are diferent elemenst in $Z_l$, then $\zeta_j/\zeta_{j'}$ can not be a root of $\Phi_{p^{\beta_n}}$. So, the elements in $Z_l$ are roots of $F_{0,\beta_1,\cdots,\beta_{m-1}}$ whose ratios also are roots. In consequence, by inductive hypothesis, we have that $|Z_l|\leq p^{m-1}$, and then $n=|\bigcup_{0\leq l<p} Z_l|\leq p\cdot p^{m-1}$, as desired. 
	
\end{proof}

\subsection{Maximality}
In this section we will show that the bound in Theorem \ref{orto} can be reached and is associated with a maximality property. Since we are dealing only with integer sets of frequencies, we introduce the following definition. It might sound artificial but we will use it only in auxiliary propositions.  

\begin{defn}
	Given $\Lambda\subset\Z$ we say that it is an integer orthogonal maximal set of frequencies for $\mu$ 
    if and only if the following two condition are satisfied:
    \begin{enumerate}
        \item $E(\Lambda)$ is orthogonal in $L^2(\mu)$
        \item If $\Lambda'\subset\Z$ is such that $\Lambda\subset\Lambda'$ and $E(\Lambda')$ is orthogonal in $L^2(\mu)$ then $\Lambda'=\Lambda$.
    \end{enumerate}
\end{defn}

\begin{proposition}\label{maxint}
	Let 
	$\Lambda$ be an integer orthogonal maximal set of frequencies for $\mu$. Fix $n\geq 0$ and $(d_j)_{j=0}^{n-1}$ with $0\leq d_j <N$. If   $\Lambda_n(d_0,\cdots,d_{n-1})\neq \emptyset$ then $$|\{d_n(\lambda):\lambda \in \Lambda_n(d_0,\cdots,d_{n-1})  \}|=p^{|T|}.$$ 
\end{proposition}

\begin{proof}
	Let $m$ be the cardinal of $\{d_n(\lambda): \lambda \in \Lambda_n(d_0,\cdots,d_{n-1})  \}$. As a consequence of the previous proposition, we know $m\leq p^{|T|}$. 

	Take $\lambda_0, \lambda_1,\cdots,\lambda_{m-1}$ all in $\Lambda_n(d_0,\cdots,d_{n-1})$ such that all $d_n(\lambda_j)$ are different. As pointed out in Remark \ref{ll}, we can define $\theta_{jl}\in(0,1)$ 
	\begin{equation*}
	\theta_{jl} = \frac{\lambda_j-\lambda_l}{N^n} - \lfloor \frac{\lambda_j-\lambda_l}{N^n} \rfloor = \frac{s_{jl}} {p^{t_{jl}}},
	\end{equation*}
	with $p\nmid s_{jl}$, $0\leq t_{jl}\leq \alpha$ and $m_D(\theta_{jl})=0$ if $j\neq l$. Also, since $p\nmid s_{jl}$, we have that $\Phi_{p^{t_{jl}}}(e^{2\pi i \theta_{jl}})=0$ and $t_{jl}\in T$.%

	Consider sequences of the form $(s_t)_{t\in T}$ with $0\leq s_t<p$ and not all of them zeros. There are $p^{|T|}-1$ such sequences. If it was the case that $m<p^{|T|}$, then there is one sequence such that
	\[\ell:=\sum_{t\in T}\frac{s_t}{p^t }
	\notin \left\{\frac{s_{j0}}{p^{t_{j0}}}:0<j<m \right\}.
	\] 
	
	Consider $
	\overline \lambda:= d_0+d_1N + \cdots + d_{n-1} N^{n-1} + d_n(\lambda_0) N^n + \ell N^{n+1}.
	$ It is important to observe that $\overline \lambda\in\Z$, since $\alpha$ is un upper bound of $T$. Our goal is to contradict maximality of $\Lambda$ by proving that $e_{\overline \lambda }$ is orthogonal to all the elements in $E(\Lambda)$. 
	
	First, take $\lambda \in \Lambda_n(d_0,\cdots,d_{n-1})$. There is $0\leq j<m$ such that $d_n(\lambda)=d_n(\lambda_j)$. We then have:
	\(
	\overline\lambda-\lambda= N^{n+1}\ell  + N^n \left( d_n(\lambda_0) - d_n( \lambda_j ) \right) + N^{n+1} K, 
	\) 
	for some $K\in\Z$. That means that 
	\[
	\theta := \frac{\overline\lambda-\lambda} {N^{n+1}}- K =
	\sum_{t\in T} \frac{s_t} {p^t} +
	\frac{  d_n(\lambda_0) - d_n( \lambda_j ) } {N} =
	\sum_{t\in T} \frac{s_t} {p^t} - \frac{s_{j0}}{p^{t_{j0}}}.
	\]
	Note that $\theta$ can be written as $\sum_{t\in T} \frac{\tilde {s_t}} {p^t}$ with $-p< \tilde{s_t} <p$ and at least one $\tilde {s_t}\neq 0$. If $\overline t$ is the biggest $t\in T$ for wich $\tilde s_t\neq 0$, then $\Phi_{p^{\overline t}}(e^{2\pi i \theta})=0$. Since $\Phi_{p^{\overline t}}\mid P_D$, we can conclude that $m_D(\frac{\overline\lambda-\lambda} {N^{n+1}})=0$. In consequence, $\hat\mu(\overline \lambda -\lambda)=0$. We have that $e_{\overline\lambda}$ and $e_{\lambda}$ are orthogonal.
    
	Now, take $\lambda\in\Lambda$ but $\lambda\notin  \Lambda_n(d_0,\cdots,d_{n-1})$. This means that  for some $0\leq k <n$ it must hold $d_k(\lambda)\neq d_k$. Take $n'$ the minimum such $k$. We have picked $0\leq n'<n$ such that $d_k(\lambda)=d_k$ if $0\leq k<n'$ and $d_{n'}\neq d_{n'}.$
    
    By hypothesis, we have $\Lambda_n(d_0,\cdots,d_{n-1})\neq \emptyset$, therefore we can pick $\tilde \lambda$ in this set. 
    In particlar, $d_{n'}(\tilde\lambda) \neq d_{n'}(\lambda) $
    and $d_{k}(\tilde\lambda) = d_{k}(\lambda)$ for $k<n'$.  
    
    Since $\lambda,\tilde\lambda$ are in $\Lambda$ then $e_{\lambda}$ and $e_{\tilde \lambda}$ are orthogonal. We can, then, proceed as before to obtain $\tilde\theta\in(0,1)$ and $t\in T$ for which $\Phi_{p^t}(e^{2 \pi i \tilde\theta})=0$, where $\tilde\theta$ is given by:
	\[
	\tilde \theta := \frac{  d_{n'}(\tilde \lambda) - d_{n'}( \lambda ) } {N} .
	\] 
    Furthermore, $\Phi_{p^t}(e^{2 \pi i (\frac{\tilde\lambda - \lambda}{N^{n'+1}})})=0$ and $m_D(\frac{ \tilde \lambda -  \lambda  } {N^{n'+1}})=0$.  Observe that $\overline\lambda$ and $\tilde\lambda$ have the first $n$ digits equal to $d_0,\dots,d_{n-1}$, then we can conclude also $m_D(\frac{\overline \lambda - \lambda}{N^{n'+1}})=0$. This completes the proof that $e_{\overline\lambda}$ is orthogonal to $e_\lambda$ for every $\lambda\in Lambda$. This would then make $\Lambda\cup\{ \overline\lambda \}$ an orthogonal set of frequencies, contradicting the maximality of $\Lambda$. The contradiction follows from the assumption $m<p^{|T|}$
\end{proof}

As we said before, integer maximal orthogonality is forced. We are interested, instead, in the usual notion of maximality.

\begin{defn}
	Given $\Lambda\subset\R$ we say that it is an orthogonal maximal set of frequencies for $\mu$ if $E(\Lambda)$ is orthogonal in $L^2(\mu)$ and for any $\Lambda'$ such that $\Lambda\subset\Lambda'$ with $E(\Lambda')$ orthogonal in $L^2(\mu)$ we necessarily have $\Lambda'=\Lambda$. 
\end{defn}

Given a real or complex function $f$, lets denote by $Z(f)$ its zero set: \(
Z_{\C}(f):=\{ z\in \C: f(z)=0 \}
\) or 
\(
Z_{\R}(f):=\{ t\in \R: f(t)=0 \}.
\)
\begin{theorem}\label{homog}
    Assume $Z_{\C}(P_D)\cap S^1\subset \bigcup_{t\in T}Z_{\C}(\Phi_{p^t})$. 
    Let $\Lambda\subset \R$ with $0\in \Lambda$. If $\Lambda$ is an orthogonal maximal set of frequencies for $\mu$ and    
    if   $\Lambda_n(d_0,\cdots,d_{n-1})\neq \emptyset$ then $$|\{d_n(\lambda):\lambda \in \Lambda_n(d_0,\cdots,d_{n-1})  \}|=p^{|T|}.$$ 
\end{theorem}

\begin{proof}
    Note that any orthogonal set of frequencies containing zero need to be contained in $Z_{\R}(\hat\mu)$. On the other hand, in virtue of \eqref{producto}, we have 
    $$
    Z_{\R}(\hat\mu)=\bigcup_{n\geq q}N^nZ_{\R}(m_D).
    $$ 
    Recall that $\xi \in Z_{\R}(m_D)$ if and only if $e^{2\pi i \xi}\in Z_{\C}(P_D)$ and by hypothesis this are also zeros of $\Phi_{p^t}$ for some $t\in T$. Also, observe that $e^{2\pi i \xi}\in Z_{\C}(\Phi_{p^t})$ if and only if $\xi=\frac{s}{p^t}$ for $s\in \Z$ with $p\nmid s$. In particular (since $\alpha$ is an upper bound of $T$), we have $N^n\xi\in \Z$ for any $n\geq 1$ and $\xi\in Z_{\R}(m_D).$ We conclude that $Z_{\R}(\hat\mu)\subset\Z$.
    
    Furthermore, if $\Lambda$ is an orthogonal set of frequencies containing zero, then it must be contained in the zeros of $\hat\mu$, thus, in this case in $\Z$. Now, the thesis is a corollary of Proposition \ref{maxint}. 
\end{proof}

\subsection{Hadamard}
There is a crucial condition regarding orthogonality, which we will call Hadamard. We give here the definition of Hadamard pair and Hadamard triple. We point out that Laba and Wang in \cite{LW02} use the terminology `compatible pair', instead of `Hadamard pair'. 
\begin{defn}\label{Had_pair}
Given two finite sets $A$ and $B$, with $|A|=|B|$, we say that $(A,B)$ is a Hadamard pair if the matrix 
$(e(a\cdot b))_{\substack{{a\in A,}\\  {b\in B}}}$ is unitary. We say $(N,A,B)$ is a Hadamard triple if the matrix $(e(\frac{a\cdot b}{N}))_{\substack{{a\in A,}\\  {b\in B}}}$ is unitary.
\end{defn}

\begin{rmk}
    \begin{enumerate}
        \item The point mass measure $\sum_{a\in A}\delta _a$ is a spectral measure with spectrum $B$ if and only if $(A,B)$ is a Hadamard pair.  
        \item Let $D$ and $L$ be a finite set with $|D|=|L|$. It holds that $m_D(\frac{l-l'}{N})=0$ for any $l\neq l'\in L$ if and only if $(N,D,L)$ is a Hadamard triple. 
    \end{enumerate}
\end{rmk}
Next Theorem shows that for certain Cantor measures, maximal orthogonal sets of frequencies correspond with digits in Hadamard triples.

\begin{theorem}\label{Hadamard}
    Assume $P_D= \prod_{t\in T}\Phi_{p^t}$.  If $\Lambda$ is an orthogonal maximal set of frequencies for $\mu$ with $0\in\Lambda$ and    
    if   $\Lambda_n(d_0,\cdots,d_{n-1})\neq \emptyset$ then: $$|\{d_n(\lambda):\lambda \in \Lambda_n(d_0,\cdots,d_{n-1})  \}|=|D|.$$ 
    Moreover, if $L=\{d_n(\lambda):\lambda \in \Lambda_n(d_0,\cdots,d_{n-1})  \}$ then $(N,D,L)$ is a Hadamard triple. 
\end{theorem}

\begin{proof}
    With these hypotheses, we already know that $|\{d_n(\lambda):\lambda \in \Lambda_n(d_0,\cdots,d_{n-1})  \}|=p^{|T|}$ by Theorem \ref{homog}. On the other hand, $|D|=P_D(1)$ and $\Phi_{p^t}(1)=p$. Since, by hypothesis, he have $P_D(1)=\prod_{t\in T}\Phi_{p^t}(1)$ we conclude $|D|=p^{|T|}$ which proves the first part of the theorem.
    
    Now, from Remark \ref{ll}, we know that $m_D(\frac{l-l'}{N})=0$ for any $l\neq l'$ in $\{d_n(\lambda):\lambda\in\Lambda_n(d_0,\cdots,d_{n-1})\}$. This is equivalent to saying that columns $l$ and $l'$ in the matrix $(e^{2\pi i d\cdot l/N})_{d\in D,\  l\in L}$ are orthogonal. Since $|\{d_n(\lambda):\lambda\in\Lambda_n(d_0,\cdots,d_{n-1})\}|=|D|$ this implies that the matrix is unitary and, hence, the thesis follows. 
    
\end{proof}

\section{Trees}\label{tree}
We can consider the rooted $N$-homogeneous tree and assign to each vertex a label $\ell$ with $0\leq \ell <N$. If the label of a vertex of order $n$ is associated with the $n$-th digit in base-$N$ expansion of integers, then there is a correspondence between infinite paths in the tree ending only with 0's or only with $N-1$ and $\Z$.
In the previous section, we proved that for an integer orthogonal set of frequencies, the number of possible values for the $n$-th digit in the base-$N$ expansion (for elements with the first $n-1$ digits prescribed) are bounded by $p^{|T|}$. Thus, for an integer orthonormal set of frequencies, a subtree with homogeneity $p^{|T|}$ instead of $N$, suffices. Moreover, under the hypothesis of Theorem \ref{Hadamard}, there is a correspondence between infinite paths on this subtree and frequencies in a maximal orthogonal set. Let us be more specific.

Recall some basic definitions on graphs and trees. A graph is a pair $(V,E)$ where $V$ is a discrete set of vertices and $E\subset V\times V$ is the collection of edges. A path (finite or infinite) is a sequence of vertices $(v_n)$ such that $(v_n,v_{n+1})\in E$ for each $n$. A tree is a connected and loop free graph (a loop is a finite path $(v_1,\dots,v_n)$ with $v_1=v_n$) and a rooted tree is a tree with one specific vertex designated as the ``root". In rooted trees, a vertex is said of order $n$, if its distance to the root is $n$, that is, if a path of length $n$ joins the vertex to the root. A rooted tree is said $m$-homogeneous if each vertex has $m$ edges connecting to vertices in the next level. Thus, in an $m$-homogeneous rooted tree, there are $m^n$ vertices of level $n$. 

To each vertex of the rooted $|D|$-homogeneous tree assign a label from the set $\{0,\dots,N-1\}$. This assignment is called a labeling. A labeling is said $(N,D)$-spectral (or just spectral if it is clear by context) if the following conditions are satisfied:
\begin{enumerate}
    \item \label{uno} There is an infinite path starting at the root with each vertex labeled 0. 
    \item \label{dos} For any vertex $v$, there is (at least) an infinite path starting from $v$ and ending by nodes all labeled $0$ or all labeled $N-1$.
    \item \label{tres} For any vertex $v$ of order $n$, let $L=L_v$ be the set of labels of the vertex in level $n+1$ sharing an edge with $v$. For all $v\in V $, $(N,D,L)$ is a Hadamard triple. 
\end{enumerate}

Let $\cl$ be such a labeling. Denote by $\bf{T}(\cl)$ the tree labeled by $\cl$. For each infinite path $(d_n)_{n\geq 0}$ in $\bf{T}(\cl)$, if it  ends with zeros consider the integer $\sum_{n=0}^\infty d_nN^n$; and if it ends with $N-1$'s consider the integer $\sum_{n=0}^{n_0} d_nN^n - N^{n_0}$, where $d_n=N-1$ for all $n\geq n_0$. Let $\Lambda(\cl)$ be the set of all such integers. In other words, $\Lambda(\cl)$ is the set of integers whose base-$N$ expansion digits correspond to infinite paths in $\bf{T}(\cl)$ ending with only 0's or only $N-1$'s. 

\begin{lemma}
    If $\lambda$, $\lambda'$ are different elements in $\Lambda(\cl)$ for some spectral label $\cl$ then $e_\lambda$ and $e_{\lambda'}$
 are orthogonal in $L^2(\mu)$.
\end{lemma}

\begin{proof}
    Define $k=\min\{n\geq 0:d_n(\lambda)\neq d_n(\lambda')\}$. Then $l=d_{k}(\lambda)$ and $l'=d_{k}(\lambda')$ are labels of vertices of order $k$ with the same parent. In consequence, they belong to a set $L$ such that $(N,D,L)$ is a Hadamard triple. 
    Since $m_D ((l-l')/N)$ is the inner product between the $l$-th and $l'$-th columns of the matrix $(e^{2\pi i d \cdot l/N})_{d\in D, l\in L}$ we conclude that $m_D((l-l')/N)=0$. 

    In these circumstances, $m_D((\lambda-\lambda')/N^{k+1})=m_D((l-l')/N)=0$. Therefore, $\hat\mu(\lambda-\lambda')=0$, as desired.  
\end{proof}


\begin{lemma}\label{maxim}
     If $P_D= \prod_{t\in T}\Phi_{p^t}$ and $\cl$ is a spectral labeling of the rooted $|D|$-homogeneous tree then $\Lambda(\cl)$ is a maximal set of orthogonal frequencies. 
\end{lemma}

\begin{proof}
    By the previous Lemma we know that $\Lambda(\cl)$ is an orthogonal set of frequencies. 
    To see that is maximal, take $\overline\lambda$ such that $e_{\overline\lambda}$ is orthogonal to $e_{\lambda}$ for any $\lambda\in \Lambda(\cl)$. 
    By condition \ref{uno} of spectrality, $0\in \Lambda(\cl)$. Thus, $e_{\overline\lambda}$ and $e_0$ are orthogonal and we have $\hat\mu(\overline\lambda)=0$. By hypothesis, the zeros of $P_D$ are the union of the zeros of $\Phi_{p^t}$ so, as in the proof of Theorem \ref{homog}, we have $\overline\lambda\in\Z$. Now can expand $\overline\lambda$ in base $N$, obtaining $(d_n(\overline\lambda))_{n\geq 0}$.
    If $\overline\lambda\notin \Lambda(\cl)$, then there exists an index $k$ such that the path $d_0(\overline\lambda)\cdots d_k(\overline\lambda)$ is not in the tree $\bf{T}(\cl)$. Take the minimum such $k$. Thus, there exists $\lambda \in \Lambda(\cl)$ such that $d_n(\lambda)=d_n(\overline\lambda)$ for $n<k$. 

    Note that $\overline\Lambda:=\Lambda(\cl)\cup \{\overline\lambda\}$ is an orthogonal set of frequencies. If we note by $\overline L=\{ d_k(\lambda):\lambda \in \overline\Lambda_k(d_0(\overline\lambda),\cdots,d_{k-1}(\overline\lambda)) \}$, then by Theorem \ref{orto}, $|\overline L|\leq p^{|T|}$. Recall that $p^{|T|}=|D|$ by hypothesis. On the other hand we have: 
    $$
    \overline L = \{ d_k(\overline\lambda ) \} \cup \{d_k(\lambda):\lambda\in \Lambda(\cl)_k(d_0(\overline\lambda),\cdots,d_{k-1}(\overline\lambda))\}.$$
    Therefore, $|\overline L|=|D|+1$ by construction. This contradicts that $|\overline L|\leq |D|$.  
    
\end{proof}

\begin{theorem}\label{theotree}
    Assume $P_D=\prod_{t\in T}\Phi_{p^t}$. If $\Lambda$ is a maximal orthogonal set of frequencies with $0\in\Lambda$ then there exists a spectral labeling $\cl$ such that $\Lambda=\Lambda(\cl)$. Reciprocally, if $\cl$ is a spectral labeling then $\Lambda(\cl)$ is a maximal orthogonal set containing $0$.  
\end{theorem}

\begin{proof}
    Assume $\Lambda$ is an orthogonal maximal set of frequencies containing zero and let us define a labeling $\cl$. It will be done inductively in the order of vertices. By Theorem \ref{Hadamard}, the set $\{ d_0(\lambda):\lambda\in\Lambda \}$ has $|D|$ elements, so we can label the vertices of level 1 with them. Assume the vertices of the first $n$ levels were labeled in such a way that if a path from the root to one vertex of level $n$ was labeled $d_0\cdots d_{n-1}$ then $\Lambda_n (d_0,\cdots,d_{n-1})\neq\emptyset$. Then, if $v$ is a vertex of level $n$ and a path from the root to $v$ was labeled $d_0\cdots d_{n-1}$ then $\Lambda_n(d_0\cdots d_{n-1})\neq\emptyset$.
    In this case, by Theorem \ref{Hadamard}, the set $\{d_n(\lambda): \lambda\in \Lambda_n(d_0\cdots d_{n-1})\}$ has $|D|$ elements. These elements will label the children of the vertex $v$. In this way, the vertices of level $n+1$ were labeled satisfying that if a path from the root to a vertex of level $n+1$ was label $d_0\cdots d_{n}$ then $\Lambda_{n+1}(d_0\cdots d_n)\neq\emptyset$. Thus, the induction is consistent and it was defined a labeling $\cl$. Let us see that is spectral. 
    
    Observe that for each element $\lambda\in\Lambda$ the sequence $(d_n(\lambda))_{n\geq 0}$ corresponds to an infinite path in the tree $\bf{T}(\cl)$. In consequence, since $0\in\Lambda$ there is a path labeled only with 0's. Condition \ref{uno} is satisfied. Moreover, if $v$ is a vertex of order $n$ and its path from the root was labeled $d_0\cdots d_{n-1}$ then there is a $\lambda\in \Lambda$ with $d_k=d_k(\lambda)$ and then the path on the tree corresponding to $\lambda$ passes through $v$ an ends only with 0's or only with $N-1$'s. Condition \ref{dos} is satisfied. Condition \ref{tres} is immediate from the construction of $\cl$. Thus $\cl$ is a spectral labeling. 

    Furthermore, this proves that $\Lambda\subset\Lambda(\cl)$. Since $\Lambda$ is maximal and $\Lambda(\cl)$ is orthogonal, we conclude $\Lambda=\Lambda(\cl). $
    
    The converse is just Lemma \ref{maxim}. 
\end{proof}

\subsection{Example}
If $N=8$ and $D=\{0,2,4,6\}$ then $P_D(x) = \Phi_4(x) \Phi_8(x)$. So, it satisfies the hypothesis of Theorem \ref{theotree}. In figure \ref{Arbol} is pictured a possible labeling for the tree in  first levels. 

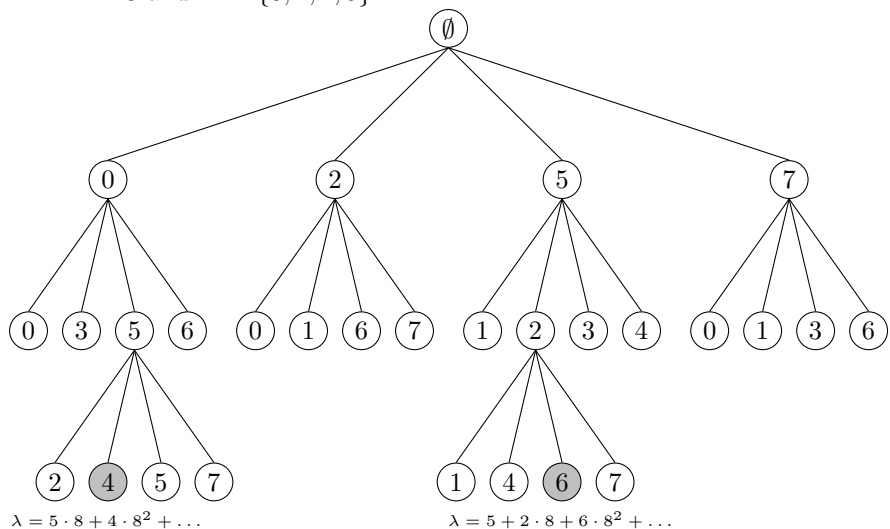
\begin{figure}[!hbt]
    \caption{First levels of a possible labeling of a spectrum for $N=8$ and $D=\{0,2,4,5\}$ }
    \label{Arbol}

\begin{tikzpicture}[
    level distance=2cm,
    level 1/.style={sibling distance=3cm},
    level 2/.style={sibling distance=.7cm},
    every node/.style={circle, draw, inner sep=1pt, minimum size=5mm}
]

\node {$\emptyset$}
    child { node {0}
        child { node {0} }
        child { node {3} }
        child { node {5}
                child { node{2} }
                child { node [fill=lightgray]{4} 
                                child[level distance=.5cm] { node[draw=none, font=\tiny, anchor=center] {$\lambda=5\cdot 8+4\cdot 8^2+\dots$}
                                edge from parent[draw=none]}
                        }
                child { node{5} }
                child { node{7} }
                }
        child { node {6} }
    }
    child { node {2}
        child { node {0} }
        child { node {1} }
        child { node {6} }
        child { node {7} }
    }
    child { node {5}
        child { node {1} }
        child { node {2} 
            child {node {1} }
            child {node {4} }
            child {node(A) [fill=lightgray]{6} 
                    child[level distance=.5cm] { node[draw=none, font=\tiny, anchor=center] {$\lambda=5+2\cdot 8+6\cdot 8^2+\dots$}
                                edge from parent[draw=none]}
                }
            child {node {7} }
            }
        child { node {3} }
        child { node {4} }
    }
    child { node {7}
        child { node {0} }
        child { node {1} }
        child { node {3} }
        child { node {6} }
    };

    \end{tikzpicture}
\end{figure}

\section*{Acknowledgment} The author wants to thank Lucía Rossi for renewing my enthusiasm for research in general and for these topics in particular.


\end{document}